\documentclass[a4paper,11pt]{article}

\usepackage[top=1in,right=1in,left=1in,bottom=1in,a4paper]{geometry}

\usepackage{amsmath, amsthm, amssymb, amsfonts}

\usepackage{enumerate}
\usepackage[all]{xy}

\newdir{ >}{!/-8.5pt/@{>}}

\newtheoremstyle{theoremstyle}
  {10pt}      
  {5pt}       
  {\itshape}  
  {}          
  {\bfseries} 
  {.}         
  {.5em}      
  {}          

\newtheoremstyle{examplestyle}
  {10pt}      
  {5pt}       
  {}          
  {}          
  {\bfseries} 
  {.}         
  {.5em}      
  {}          

\newtheoremstyle{sub-style}
  {10pt}      
  {5pt}       
  {}  
  {}          
  {\bfseries} 
  {.}         
  {.5em}      
  {}          

\theoremstyle{theoremstyle}
\newtheorem{theorem}{Theorem}[section]
\newtheorem*{theorem*}{Theorem}
\newtheorem{lemma}[theorem]{Lemma}
\newtheorem*{lemma*}{Lemma}
\newtheorem{proposition}[theorem]{Proposition}
\newtheorem*{proposition*}{Proposition}
\newtheorem{corollary}[theorem]{Corollary}
\newtheorem*{corollary*}{Corollary}

\theoremstyle{examplestyle}
\newtheorem{example}[theorem]{Example}

\newtheorem{definition}[theorem]{Definition}
\newtheorem{definition*}{Definition}
\newtheorem{remark}[theorem]{Remark}
\newtheorem{remark*}{Remark}

\theoremstyle{sub-style}
\newtheorem{sub}[theorem]{}

\newcommand{\comment}[1]{}

\newcommand{\sh}[1]{\mathcal{#1}}

\newcommand{\spec}[1]{\operatorname{Spec}(#1)}

\newcommand{\Hom}{\operatorname{Hom}}
\newcommand{\HOM}{\operatorname{HOM}}

\newcommand{\ksm}{{\K[\sigma_M]}}

\newcommand{\goodquot}{/\negmedspace/}
\newcommand{\ilim}{\varprojlim}
\newcommand{\colim}{\varinjlim}

\newcommand{\E}{\mathbf{E}}
\newcommand{\K}{K}
\newcommand{\Z}{\mathbb{Z}}

\newcommand{\R}{\mathbb{R}}
\newcommand{\N}{\mathbb{N}}

\newcommand{\uc}{{\underline{c}}}

\newcommand{\mksmMod}{\text{$M$-$\ksm$-$\operatorname{Mod}$}}

\title{A lifting functor for toric sheaves}

\date{August 2012}

\author{Markus Perling\footnote{Fakult\"at f\"ur Mathematik Ruhr-Universit\"at Bochum
Universit\"atsstra\ss e 150 44780 Bochum, Germany, {\tt Markus.Perling@rub.de}}}

\begin{document}

\maketitle

\begin{abstract}
For a variety $X$ which admits a Cox ring, we introduce a functor from the category
of quasi-coherent sheaves on $X$ to the category of graded modules over the homogeneous
coordinate
ring of $X$. We show that this functor is right adjoint to the sheafification functor and therefore
left-exact. Moreover, we show that this functor preserves torsion-freeness and reflexivity.
For the case of toric sheaves, we give a combinatorial characterization of its right derived functors
in terms of certain right derived limit functors.
\end{abstract}


\section{Introduction}

Consider an affine normal variety $W = \spec{S}$ over an algebraically closed
field $\K$, $G$ a diagonalizable group scheme which acts on $W$, and
$H \subseteq G$ a closed subgroup scheme. We denote $T$ the quotient of
diagonalizable groups schemes $G / H$. Moreover, we assume the following.
\begin{itemize}
\item There exists a Zariski-open $G$-invariant subset $\hat{X}$ of $W$
such that a good quotient $X = \hat{X} \goodquot H$ exists. We denote
$\pi: \hat{X} \rightarrow X$ the corresponding projection.
\item $X$ admits an affine $T$-invariant open covering (this is
automatic if $T$ is a torus, see \cite{sum1}).
\item The complement $Z = W \setminus \hat{X}$ has codimension at least $2$.
\end{itemize}
The actions of $G$ and $H$ on $W$ induce gradings on $S$ by the character groups
$X(G)$ and $X(H)$,
respectively, which are compatible via the surjection $X(G) \twoheadrightarrow
X(H)$. With this, we require moreover the following.
\begin{itemize}
\item $X(H) \cong A_{d - 1}(X)$, the divisor class group, and for suitable
representatives $D_\chi \in A_{d - 1}$  there exists an isomorphism of
$S_0$-modules
$S \cong \bigoplus_{\chi \in X(H)} \Gamma\big(\sh{O}(D_\chi)\big)$ which
is compatible with the $X(H)$-grading of $S$. In particular, the latter carries an
induced ring structure.
 \end{itemize}
These conditions essentially imply that $X$ is a variety which admits a Cox ring,
where we admit possibly some further action by a diagonalizable group scheme.
In particular, this class of varieties includes the Mori dream spaces. The main
application we have in mind is the case where $T$ is a torus and $X$ a toric variety
such that $S$ is the associated homogeneous
coordinate ring as defined in \cite{Cox}.
It was shown in \cite[\S 6]{PerlingTrautmann} that by taking
local invariants we obtain an exact and essentially surjective functor which
maps an $X(G)$-graded $S$-module $E$ to a $T$-equivariant quasi-coherent
sheaf $\widetilde{E}$ on $X$,
the so-called {\em sheafification functor}. Conversely,
there is a functor from the category of $T$-equivariant sheaves on $X$
to the category of $X(G)$-graded $S$-modules, mapping a quasi-coherent
sheaf $\sh{E}$ to a $X(G)$-graded $S$-module $\Gamma_* \sh{E} :=
\Gamma(\hat{X}, \pi^*\sh{E})$. This functor is right inverse to the
sheafification functor, i.e., we have $\widetilde{\Gamma_* \sh{E}}�\cong
\sh{E}$ for any $\sh{E}$. However, the functor $\Gamma_*$ in many
cases is not very well behaved. So it usually does not preserve properties
such as torsion-freeness and reflexivity. Also, by being the composition of
the right-exact functor $\pi^*$ with the left-exact global section
functor, $\Gamma_*$ does not have any exactness properties. In general,
$\Gamma_*$ is right-exact if $\hat{X} = W$ (and thus $X$ is affine)
and it is left-exact if $\pi$ is a flat morphism.

The aim of this note is to construct an alternative functor to $\Gamma_*$,
which we are going to call the {\em lifting functor}, which maps a quasi-coherent
$T$-equivariant sheaf $\sh{E}$ to an $X(G)$-graded $S$-module $\widehat{\sh{E}}$.
We will show that the lifting functor has the following two general properties:
\begin{enumerate}
\item The lifting functor is right adjoint the sheafification functor and therefore left-exact
(Theorem \ref{adjointness}).
\item Lifting preserves torsion-freeness and reflexivity. For torsion free sheaves it preserves
coherence (Theorem \ref{coherencetheorem1}).
\end{enumerate}

The lifting functor is an offspring of recent work on toric sheaves \cite{perling7}. Assume that
$X$ is a toric variety and $\hat{X}$ the standard quotient presentation as in \cite{Cox}. By
results of Klyachko \cite{Kly90}, \cite{Kly91}, any coherent reflexive $T$-equivariant sheaf
$\sh{E}$ can be described by a finite-dimensional vector space together with a family of
filtrations parameterized by the rays of the fan associated to $X$. In order to represent $\sh{E}$
by an appropriate $\Z^n$-graded module over the homogeneous coordinate ring, it is a
rather straightforward observation that, rather than taking $\Gamma_* \sh{E}$, we can
choose a reflexive sheaf which is associated to precisely the same filtrations as $\sh{E}$
(this is possible because there is a one-to-one correspondence among the rays of the fans associated
to $X$ and $\hat{X}$, respectively).
Our results show that this ad-hoc observation indeed has a functorial interpretation.
In Section \ref{toricsection} we will see that the lifting functor has moreover a very nice interpretation
in the combinatorial setting of \cite{perling7}.

\section{Preliminaries}

\begin{sub}
Let $A$ be any abelian group, $S$ a $A$-graded $\K$-algebra, and $E$ an $A$-graded $S$-module.
Then $E \cong \bigoplus_{\alpha \in A} E_\alpha$ and for any $\beta \in A$ we denote
$E(\beta) = \bigoplus_{\alpha \in A} E_{\alpha + \beta}$ the {\em degree shift} of $E$ by $\beta$.
\end{sub}

\begin{sub}\label{tensorproduct}
For any two $S$-modules $E$ and $F$,
The tensor product $E \otimes_R F$ can be $A$-graded as follows. Consider first the
$\K$-vector space $E \otimes_\K F$ and set $(E \otimes_\K F)_\alpha
= \bigoplus_{\beta \in A} (E_\beta \otimes_\K F_{\alpha - \beta})$.
Then for $\alpha \in A$ we form $(E \otimes_R F)_\alpha$ as the quotient
of $(E \otimes_\K F)_\alpha$ by the subvector space generated
by $re \otimes f - e \otimes rf$ for $e \in E, f \in F, r \in R$. Note that
$E(\alpha) \otimes_R F \cong E \otimes_R F(\alpha) \cong (E \otimes_R F)(\alpha)$.
\end{sub}

\begin{sub}\label{HOM}
For any $A$-graded $S$-modules $E$, $F$,
the {\em graded} version of $\Hom_S(E, F)$ by definition is given by
\begin{equation*}
\HOM^A_S(E, F) := \bigoplus_{\alpha \in A} \Hom^A_S(E, F(\alpha)),
\end{equation*}
where $\Hom^A_S(E, F(\alpha)) = \{f \in \Hom_S(E, F) \mid
f(E_\beta) \subseteq F_{\beta + \alpha}$ for every $\beta \in A\}$.
We can consider in a natural way $\HOM^A_S(E, F)$
as a subset of $\Hom_R(E, F)$. Moreover, within the graded setting,
$\HOM^A_S$ satisfies the same general functorial properties as the standard
$\Hom$ (see \cite[\S 2]{NO2}).
Note that when we speak of the category of $A$-graded modules, then the set
of morphisms between modules $E$ and $F$ is given by $\Hom^A_S(E, F)$ and
{\em not} by $\HOM^A_S(E, F)$.
\end{sub}

\begin{sub}
We will deal with three gradings, given by the character groups $X(T)$,
$X(G)$, and $X(H)$, respectively. Any given $X(G)$-graded ring $S$
carries an $X(H)$-grading as well via the surjection $X(G) \twoheadrightarrow X(H)$.
To distinguish between these two gradings, we write the homogeneous components
$S_{(\alpha)}$ and $S_{\chi}$ for the $X(H)$- and the $X(G)$-grading, respectively,
where $\alpha \in X(H)$ and $\chi \in X(G)$.
For $\chi \in X(G)$ we may also write $S_{(\chi)}$ for the $X(H)$-homogeneous
component determined by the image of $\chi$ in $X(H)$. Then $S_{(\chi)}$ has a
natural $X(T)$-grading which is given by $S_{(\chi)} \cong \bigoplus_{\eta \in X(T)}
(S_{(\chi)})_\eta$ with $(S_{(\chi)})_\eta = S_{\chi + \eta}$. We use the same
conventions for $X(G)$- and $X(H)$-graded $S$-modules.
\end{sub}

\begin{sub}\label{gradedinclusions}
For any $X(G)$-graded $S$-modules $E, F$, we have the two graded modules
$\HOM^{X(G)}_S(F, E)$ and $\HOM^{X(H)}_S(F, E)$,
together with the natural sequence of inclusions
\begin{equation*}
\HOM^{X(G)}_S(F, E) \subseteq \HOM^{X(H)}_S(F, E) \subseteq \HOM_S(F, E)
\end{equation*}
(which even satisfy certain topological properties, \cite[\S 2.4]{NO2}).
\end{sub}

\begin{sub}\label{extendedgrading}
The $X(H)$-invariant subring $R = S_{(0)}$ is automatically $X(T)$-graded. It is
also $X(G)$-graded by trivial extension, i.e., we set $R_\chi = 0$ for every $\chi \in
X(G) \setminus X(T)$. Likewise, every $X(T)$-graded $R$-module can be given an
$X(G)$-grading.
\end{sub}

\begin{sub}\label{signconvention}
With the notation as in \ref{HOM}, note that we have
$\HOM^A_S(F, E) = \bigoplus_{\alpha \in A} \Hom^A_S(F, E(\alpha)) =$ $\bigoplus_{\alpha \in A}
\Hom^A_S(F(-\alpha), E))$. However, in order to avoid some cumbersome signs, we will usually
write expressions like
$\widehat{E} = \bigoplus_{\alpha \in A} \Hom^A_S(S(\alpha), E)$, where it is understood that
the proper grading is given by $(\widehat{E})_\alpha = \Hom^A_S(S(-\alpha), E)$.
\end{sub}

\begin{sub}\label{cover}
The sheafification functor as defined in \cite{PerlingTrautmann} maps an $X(H)$-graded
(respectively $X(G)$-graded) $S$-module $E$ to a quasi-coherent sheaf $\widetilde{E}$ over
$X$ as follows. Let open affine covers $\{U_i = \spec{R_i}\}_{i \in I}$ and
$\{\hat{U}_i = \pi^{-1}(U_i) = \spec{S_i}\}_{i \in I}$ on
$X$ and $\hat{X}$, respectively, be given, such that $U_i = \hat{U}_i \goodquot H$ (both covers
can be chosen $T$ and $G$-invariant, respectively). Then $R_i = S_i^H = (S_i)_{(0)}$ for every $i \in I$
and we can associate to every $U_i$ the $R_i$-module $\Gamma(\hat{U}_i, E)_{(0)}$, where
by abuse of notation we identify $E$ with its associated quasi-coherent sheaf over $W$.
These glue naturally to give a quasi-coherent sheaf of $\sh{O}_X$-modules. Moreover, if
the $U_i$ are choosen $T$-invariant, then the $R_i$ are $X(T)$-graded, and
both the $R_i$ and $S_i$ are $X(G)$-graded by \ref{extendedgrading}. In this case,
$\sh{E}$ has also an induced $T$-equivariant structure.
\end{sub}

\section{The right adjoint}\label{rightadjointsection}

For a given morphism of schemes $f : U \rightarrow V$ and a quasi-coherent
sheaf $\sh{F}$ on $V$, it is standard to define the pullback $f^*\sh{F}$
as $f^{-1}\sh{F}
\otimes_{f^{-1} \sh{O}_V} \sh{O}_U$. This defines a right-exact functor
from the category of (quasi-)coherent $\sh{O}_V$-modules to the category
of (quasi-)coherent $\sh{O}_U$-modules. However, this is not the only conceivable
way to define a pull-back functor; instead, one could consider the sheaf
\begin{equation*}
f\hat{\ } \sh{F} := \mathcal{H}om_{f^{-1}\sh{O}_V}(\sh{O}_U, f^{-1}\sh{F}).
\end{equation*}
Clearly, $f\hat{\ }$ is a left-exact functor from the category
of quasi-coherent $\sh{O}_V$-modules to the category of quasi-coherent
$\sh{O}_U$-modules. In the affine case, i.e., $U = \spec{A}$,
$V = \spec{B}$ for some commutative rings $A$, $B$, and $\sh{F}$ the
sheaf corresponding to an $B$-module
$F$, $f\hat{\ }\sh{F}$ corresponds to the module $\Hom_B(A, F)$, where
the $A$-module structure is given by $(rg)(r') = g(rr')$ for $r, r' \in R$
and $g \in \Hom_B(A, F)$.

However, the following example shows that $f\hat{\ }$ in general
will behave quite pathological.

\begin{example}
Assume $R = F = \K$ and $T = \K[x]$. Then we have isomorphisms of
$\K$-vector spaces
\begin{equation*}
\Hom_\K(\K[x], \K) \cong \Hom_\K(\bigoplus_{i \geq 0}�\K, \K)
\cong \prod_{i \geq 0} \Hom_\K(\K, \K) \cong \prod_{i�\geq 0} \K.
\end{equation*}
That is, from a one-dimensional $\K$-vector space we have created a
$\K[x]$-module which has torsion elements and no countable generating set.
\end{example}

We will see that one can define a better behaved
graded version of this pull-back.
Under our general assumptions on $X$ and $\hat{X}$,
let $\{U_i\}_{i \in I}$ and $\{\hat{U}_i = \pi^{-1}(U_i)\}_{i \in I}$
be affine $T$- and $G$-invariant covers, respectively, as in \ref{cover}.
By the general properties of good quotients, the $\hat{U}_i$ form an
affine open covering of $\hat{X}$ such that $U_i = \hat{U}_i \goodquot H$ for every
$i \in I$. We denote $U_i = \spec{R_i}$ and $\hat{U}_i = \spec{S_i}$; then
$R_i = S_i^H$ for every $i \in I$. Moreover, both the $R_i$ and $S_i$ are
$X(G)$-graded by \ref{extendedgrading}.
For any character (and thus divisor class of $X$) $\alpha \in X(H)$, there is naturally
associated the module $\sh{O}(\alpha) \cong \widetilde{S(\alpha)}$, which is reflexive and
of rank one. This module is a  distinguished representative for the isomorphism class of
such sheaves associated to the class $\alpha$. Similarly, if we choose some $\chi \in X(G)$
which maps to $\alpha$ via the
surjection $X(G) \twoheadrightarrow X(H)$, we obtain an induced $T$-equivariant structure
on $\sh{O}(\alpha)$, which we denote by $\sh{O}(\chi) \cong \widetilde{S(\chi)}$.

\begin{definition}
Let $\sh{E}$ be a $T$-equivariant quasi-coherent sheaf on $X$. Then we set
\begin{equation*}
\sh{E}^H := \bigoplus_{\alpha \in X(H)} \mathcal{H}om_{\sh{O}_X}(\sh{O}(\alpha), \sh{E}).
\end{equation*}
and
\begin{equation*}
\sh{E}^G := \bigoplus_{\chi \in X(G)} \mathcal{H}om_{\sh{O}_X}^T(\sh{O}(\chi), \sh{E}).
\end{equation*}
\end{definition}

We first show the following.

\begin{proposition}\label{firstprop}
Let $\sh{E}$ be a $T$-equivariant quasi-coherent sheaf on $X$.
\begin{enumerate}[(i)]
\item\label{firstpropi} Both $\sh{E}^H$ and $\sh{E}^G$ are quasi-coherent subsheaves
of $\pi\hat{\ }\sh{E}$, and
$\sh{E}^H \cong \sh{E}^G$ as $\sh{O}_{\hat{X}}$-modules.
\item\label{firstpropiii} $\sh{O}_X^H$ (and therefore $\sh{O}_X^G$) is isomorphic to
$\sh{O}_{\hat{X}}$.
\item\label{firstpropii} If $\Gamma(U_i, \sh{E})$ is a first syzygy module for every $i$,
then so is $\Gamma(\hat{U}_i, \sh{E}^H)$.
\item\label{firstpropiv} If $\sh{E}$ is coherent and torsion free, then
$\sh{E}^H$ and $\sh{E}^G$ are coherent and torsion free as well.
\end{enumerate}
\end{proposition}

\begin{proof}
(\ref{firstpropi})
First note that for every $\chi \in X(G)$ which maps to $\alpha \in X(H)$,
we have a natural inclusion of sheaves of $\K$-vector spaces
$\phi_\chi: \mathcal{H}om_{\sh{O}_X}^T(\sh{O}(\chi), \sh{E})
\hookrightarrow \mathcal{H}om_{\sh{O}_X}(\sh{O}(\alpha), \sh{E})$. Summing over all
such characters, we get a map of sheaves
$\sum_{\eta \in X(T)}\phi_{\chi + \eta}:
\bigoplus_{\eta \in X(T)} \mathcal{H}om_{\sh{O}_X}^T(\sh{O}(\chi + \eta), \sh{E})
\rightarrow \mathcal{H}om_{\sh{O}_X}(\sh{O}(\alpha), \sh{E})$.
Locally, we denote $E_i := \Gamma(U_i, \sh{E})$ for every $U_i$ and this map
translates to an isomorphism of $R_i$-modules
$\bigoplus_{\eta \in X(T)} \Hom_{R_i}^{X(T)}((S_i)_{(\chi + \eta)}, E_i) \rightarrow
\HOM_{R_i}^{X(T)}((S_i)_{(\alpha)}, E_i)$. Because $(S_i)_{(\alpha)}$ is a finitely
generated $R_i$-module by our general assumptions, the latter is isomorphic to
 $\Hom_{R_i}((S_i)_\alpha, E_i)$ (see \cite[Cor. 2.4.4]{NO2}). So,
$\phi_\alpha$ is indeed an isomorphism and by summing over all $\chi \in X(G)$,
we get an isomorphism $\sum_{\chi \in X(G)} \phi_\chi: \sh{E}^G \rightarrow \sh{E}^H$.
Now, $\Gamma(\hat{U}_i, \sh{E}^H) \cong \bigoplus_{\alpha \in X(H)}
\Hom_{R_i}((S_i)_{(\alpha)}, E_i)$ and therefore $\sh{E}^H$ (and thus $\sh{E}^G$)
is quasi-coherent. Moreover, observe that locally we have $\Gamma(\hat{U}_i, \pi\hat{\ }\sh{E})
\cong \Hom_{R_i}(S_i, E) \cong \Hom_{R_i}(\bigoplus_{\alpha \in X(H)}(S_i)_{(\alpha)}, E_i)$
$\supseteq \bigoplus_{\alpha \in X(H)}\Hom_{R_i}((S_i)_{(\alpha)}, E_i)$, so $\sh{E}^H$
(and thus $\sh{E}^G$) indeed is a subsheaf of $\pi\hat{\ }\sh{E}$.

(\ref{firstpropiii}) It suffices to show that for any $i$ the module
$\hat{R}_i := \bigoplus_{\alpha \in X(H)}\Hom_{R_i}((S_i)_{(\alpha)}, R_i)$ is
naturally isomorphic to $S_i$. For this, we observe that for every $\alpha \in X(H)$ holds
$(\hat{R}_i)_{(\alpha)} = \Hom_{R_i}((S_i)_{(-\alpha)}, R_i) \cong (S_i)_{(\alpha)}$,
as the $(S_i)_{(\alpha)}$ are reflexive modules of rank one by our general assumptions.
Therefore we have natural isomorphisms $\hat{R}_i \cong \bigoplus_{\alpha \in X(H)}
(\hat{R}_i)_{(\alpha)} \cong \bigoplus_{\alpha \in X(H)} (S_i)_{(\alpha)}$ $\cong S_i$.

(\ref{firstpropii}) By assumption, we can represent $E_i$ as a first syzygy $0 \rightarrow E_i
\rightarrow R_i^{\oplus I}$, where $I$ is some index set. Applying $\bigoplus_{\alpha \in X(H)}
\Hom_{R_i}((S_i)_{(\alpha)}, \ )$ preserves left-exactness and direct sums in the right
argument, and so we obtain
an exact sequence $0 \rightarrow \hat{E}_i \rightarrow \hat{R}_i^{\oplus I} \cong S_i^{\oplus I}$,
where $\hat{E}_i := \bigoplus_{\alpha \in X(H)}\Hom_{R_i}((S_i)_{(\alpha)}, E_i) \cong
\Gamma(\hat{U}_i, \sh{E}^H)$, and
the latter isomorphism follows from (\ref{firstpropiii}).

(\ref{firstpropiv})
It suffices to show that for any $i$ the module
$\hat{E}_i := \bigoplus_{\alpha \in X(H)}\Hom_{R_i}((S_i)_{(\alpha)}, E_i)$ is torsion free.
Because $E_i$ is by assumption torsion free and finitely generated, it can be represented
as a first syzygy module $0 \rightarrow E_i \rightarrow R_i^{n_i}$ for some integer $n_i$.
Applying (\ref{firstpropii}), we obtain an exact sequence $0 \rightarrow \hat{E}_i \rightarrow
S_i^{n_i}$. Hence, $\hat{E}_i$ is finitely generated and torsion free.
\end{proof}

We come now to our main definition.

\begin{definition}\label{liftdef}
Let $\sh{E}$ be a $T$-equivariant quasi-coherent sheaf on $X$. Then we call the $X(G)$-graded $S$-module
\begin{equation*}
\widehat{\sh{E}} := \Gamma(\hat{X}, \sh{E}^G)
\end{equation*}
the {\em lifting} of $\sh{E}$.
\end{definition}

\begin{remark}\label{liftingdef}
Note that the $S$-module $\widehat{\sh{E}}$ carries both a $X(G)$-grading as well as a $X(H)$-grading,
which are given by
\begin{equation*}
\widehat{\sh{E}} \cong \bigoplus_{\alpha \in X(H)} \Hom_{\sh{O}_X}(\sh{O}(\alpha), \sh{E})
\cong \bigoplus_{\chi \in X(G)} \Hom_{\sh{O}_X}^T(\sh{O}(\chi), \sh{E})
\end{equation*}
(see \ref{signconvention} for our convention on the grading).
By construction, lifting is functorial and left-exact.
Moreover, if $X$ is smooth, then every sheaf of the form $\sh{O}(\alpha)$ is invertible and
we have natural isomorphisms $\Hom_{\sh{O}_X}(\sh{O}(\alpha), \sh{E}) \cong \Gamma(X, \sh{E}
\otimes_{\sh{O}_X} \sh{O}(-\alpha))$ for every $\alpha$. (respectively
$\Hom_{\sh{O}_X}^T(\sh{O}(\chi), \sh{E}) \cong \Gamma(X, \sh{E}
\otimes_{\sh{O}_X} \sh{O}(-\chi))^T$ for every $\chi$). In this case, our lifting
functor is naturally equivalent to the usual lifting functor $\Gamma_*$.
\end{remark}

\begin{proposition}\label{leftinverse}
The sheafification functor is left-inverse to the lifting functor.
\end{proposition}

\begin{proof}
We show that $(\widehat{\sh{E}})\tilde{\ } \cong \sh{E}$ for any $T$-equivariant quasi-coherent
sheaf on $X$. The corresponding statement about morphisms then will be evident.
With notation as in the proof of Proposition \ref{firstprop}, we have for every $i \in I$
\begin{equation*}
\Gamma(U_i, (\widehat{\sh{E}})\tilde{\ }) = \HOM^{X(G)}_{R_i}(S_i, E_i)_{(0)}
= \Hom^{X(G)}_{R_i}((S_i)_{(0)}, E_i) = \Hom^{X(T)}_{R_i}(R_i, E_i) \cong E_i.
\end{equation*}
By naturality, the $E_i$ glue to yield $\sh{E}$.
\end{proof}

Before we can prove our main result, we need to clarify how homomorphism spaces are
related under going back and forth under lifting and sheafification.

\begin{lemma}\label{maplemma}
\begin{enumerate}[(i)]
\item\label{maplemmai} For any $X(G)$-graded $S$-module $E$ there exists a natural homomorphism
of $X(G)$-graded $S$-modules $E \rightarrow \widehat{\tilde{E}}$.
\item\label{maplemmaii} Let $\sh{E}$, $\sh{F}$ be $T$-equivariant quasi-coherent sheaves on $X$. Then sheafification induces a surjective homomorphism of $\K$-vector spaces
\begin{equation*}
\Hom^{X(G)}_S(\widehat{\sh{E}}, \widehat{\sh{F}}) \twoheadrightarrow \Hom^T_{\sh{O}_X}(\sh{E}, \sh{F}).
\end{equation*}
\item\label{maplemmaiii} Let $E$ be a $X(G)$-graded $S$-module and $\sh{F}$ be a quasi-coherent sheaf on $X$.
Then sheafification induces a surjective homomorphism of $K$-vector spaces
\begin{equation*}
\Hom^{X(G)}_S(E, \widehat{\sh{F}}) \rightarrowtail \Hom^T_{\sh{O}_X}(\tilde{E}, \sh{F}).
\end{equation*}
\end{enumerate}
\end{lemma}

\begin{proof}
(\ref{maplemmai}) Degree-wise we define a map $\phi_{\chi}: E_{(\chi)} \rightarrow \Hom_S^{X(G)}(S_{(-\chi)},
E_{(0)})$ for $\chi \in X(G)$ by setting $(\phi_{\chi}(e))(s) := s \cdot e$ for every $s \in S_{(-\chi)}$.
We leave it to the reader to check that this indeed yields a $X(G)$-homogeneous homomorphism of $S$-modules.

(\ref{maplemmaii}) By revisiting the constructions of the proof of
Proposition \ref{leftinverse}, we conclude that
the functorially induced composition
\begin{equation*}
\Hom^T_{\sh{O}_X}(\sh{E}, \sh{F}) \longrightarrow \Hom^{X(G)}_S(\widehat{\sh{E}}, \widehat{\sh{F}})
\longrightarrow \Hom^T_{\sh{O}_X}(\sh{E}, \sh{F})
\end{equation*}
is a natural isomorphism. In particular, the second homomorphism is surjective.

(\ref{maplemmaiii}) By (\ref{maplemmai}) we obtain a homomorphism of $S$-modules
$\Hom^{X(G)}_S(\widehat{\tilde{E}}, \widehat{\sh{F}}) \rightarrow
\Hom^{X(G)}_S(E, \widehat{\sh{F}})$ which naturally commutes with the maps
$\Hom^{X(G)}_S(\widehat{\tilde{E}}, \widehat{\sh{F}}) \rightarrow$
$\Hom^T_{\sh{O}_X}(\tilde{E}, \sh{F})$ and $\Hom^{X(G)}_S(E, \widehat{\sh{F}})$ $\rightarrow
\Hom^T_{\sh{O}_X}(\tilde{E}, \sh{F})$, respectively, which are induced by sheafication.
By (\ref{maplemmaii}), the first map is surjective, hence the second must be surjective, too.
\end{proof}

The can now show our main results, which in particular implies that lifting is left-exact.

\begin{theorem}\label{adjointness}
The lifting functor from the category of $T$-equivariant quasi-coherent sheaves on $X$ to
the category of $X(G)$-graded $S$-modules is right adjoint to the sheafification functor.
\end{theorem}

\begin{proof}
We first consider the affine situation and assume that $\hat{X} = \spec{S}$ and $X = \spec{R} = \spec{S_{(0)}}$.
Denote $E$ an $X(G)$-graded $S$-module and $F$ an $X(T)$-graded (and therefore $X(G)$-graded, see
\ref{extendedgrading}) $R$-module.
For simplicity, we write $\widehat{F}$ for the lifting of $F$.
Then we have the isomorphisms of $X(G)$-graded $R$-modules
\begin{align*}
\HOM^{X(G)}_S(E, \widehat{F}) & = \HOM^{X(G)}_S\big(E, \HOM^{X(G)}_R(S, F)\big)\\
& \cong \HOM^{X(G)}_R(E \otimes_S S, F) \cong \HOM^{X(G)}_R(E, F).
\end{align*}
Taking invariants with respect to the $X(G)$-grading, we get
\begin{equation*}
\HOM^{X(G)}_S(E, \widehat{F})_{(0)} = \HOM^{X(G)}_R(E, F)_{(0)} = \Hom_R^{X(G)}(E_0, F)
= \Hom_R^{X(T)}(E_0, F),
\end{equation*}
where the second equality follows form the fact that $F$ is concentrated in $X(H)$-degree zero.

For the general case, consider a $T$-equivariant sheaf $\sh{F}$ on $X$ and an $X(G)$-graded
$S$-module $E$ whose restriction to $\hat{X}$ corresponds to a $G$-equivariant quasi-coherent
sheaf $\sh{E}$. As above, denote $\{U_i\}_{i \in I}$, $\{\hat{U}_i\}_{i \in I}$ a $T$-invariant
(resp. $G$-invariant) affine cover of $X$ (resp. $\hat{X}$). The affine case considered
before corresponds to isomorphisms
$\Gamma\big(\hat{U}_i, \mathcal{H}om^G_{\sh{O}_{\hat{U}_i}}(\sh{E}\vert_{\hat{U}_i},
\sh{F}^G\vert_{\hat{U}_i})\big) \rightarrow
\Gamma\big(U_i, \mathcal{H}om^T_{\sh{O}_{U_i}}(\widetilde{E}\vert_{U_i},
\sh{F}\vert_{U_i})\big)$ for every $i \in I$. These isomorphisms commute naturally with the restrictions
\begin{equation*}
\Gamma(\hat{U}_i, \mathcal{H}om^G_{\sh{O}_{\hat{U}_i}}(\sh{E}\vert_{\hat{U}_i},
\sh{F}^G\vert_{\hat{U}_i})) \rightarrow \Gamma(\hat{U}_i \cap \hat{U}_j,
\mathcal{H}om^G_{\sh{O}_{\hat{U}_i \cap \hat{U}_j}}(\sh{E}\vert_{\hat{U}_i \cap \hat{U}_j},
\sh{F}^G\vert_{\hat{U}_i \cap \hat{U}_j}))
\end{equation*}
and
\begin{equation*}
\Gamma\big(U_i, \mathcal{H}om^T_{\sh{O}_{U_i}}(\widetilde{E}\vert_{U_i}, \sh{F}\vert_{U_i})\big)
\rightarrow \Gamma\big(U_i \cap U_j, \mathcal{H}om^T_{\sh{O}_{U_i \cap U_j}}(\widetilde{E}\vert_{U_i \cap U_j},
\sh{F}\vert_{U_i \cap U_j})\big),
\end{equation*}
respectively for $i, j \in I$. Therefore we obtain an induced homomorphism
\begin{equation*}
\Hom^G_{\sh{O}_{\hat{X}}}(\sh{E}, \sh{F}^G) = \Gamma\big(\hat{X}, \mathcal{H}om^G_{\sh{O}_{\hat{X}}}
(\sh{E}, \sh{F}^G)\big) \rightarrow \Gamma\big(X, \mathcal{H}om^T_{\sh{O}_X}(\widetilde{E}, \sh{F})\big) =
\Hom^T_{\sh{O}_X}(\widetilde{E}, \sh{F}).
\end{equation*}
By the naturality of the local isomorphisms and the property  that 
$\mathcal{H}om^G_{\sh{O}_{\hat{X}}} (\sh{E}, \sh{F}^G)$ is a sheaf it follows that this
homomorphism is an isomorphism.
It remains to show that $\Hom_S^{X(G)}(E, \widehat{\sh{F}}) = \Hom^G_{\sh{O}_W}(E, \widehat{\sh{F}})$ equals
$\Hom^G_{\sh{O}_{\hat{X}}}(\sh{E}, \sh{F}^G)$. For this, consider the commutative diagram
\begin{equation*}
\xymatrix{
\Hom^G_{\sh{O}_W}(E, \widehat{\sh{F}}) \ar[d]_\phi \ar[rd]^\psi \\
\Hom^G_{\sh{O}_{\hat{X}}}(\sh{E}, \sh{F}^G) \ar[r]_{\cong} & \Hom^T_{\sh{O}_X}(\widetilde{E}, \sh{F}),
}
\end{equation*}
where $\phi$ is the restriction map and $\psi$ the map induced by the sheafification functor.
$\phi$ is injective because $\widehat{\sh{F}}$ is an extension of $\sh{F}^G$ from $\hat{X}$
to $W$ and therefore does not have torsion with support on $Z$. Now, $\psi$ is surjective by
Lemma \ref{maplemma} (\ref{maplemmaiii}), hence both $\phi$ and $\psi$ are isomorphisms.
\end{proof}

\begin{remark}
From the proofs of \ref{leftinverse} and \ref{adjointness}, it follows that the counit of the
adjunction is for every $T$-equivariant quasicoherent sheaf $\sh{E}$ given by the natural map
$(\widehat{\sh{E}})\tilde{\ } \rightarrow \sh{E}$ which, using notation from the proof of
\ref{leftinverse}, is locally given by the natural isomorphisms $\Hom^{X(T)}_{R_i}(R_i, E_i)
\overset{\equiv}{\longrightarrow} E_i$. This is an interesting observation, as it implies that
the category of $T$-equivariant sheaves on $X$ is a reflective localization of the category of
$X(G)$-graded $S$-modules by the kernel of the sheafification functor. This was previously
only known for the case where $X$ is smooth. As was pointed out to
me by M. Barakat and M. Lange-Hegermann, this is relevant for current work \cite{BLH12b}
related to computational toric geometry.
\end{remark}

\section{Coherence}\label{coherencesection}

We have seen in Proposition \ref{firstprop} that a torsion free coherent sheaf $\sh{E}$ on $X$
lifts to a torsion free coherent sheaf $\sh{E}^H$ on $\hat{X}$. In this section we want to give
similar and refined criteria for the lifting $\widehat{\sh{E}}$.

\begin{sub}
By Proposition \ref{firstprop} (\ref{firstpropi}), properties such as coherence and torsion-freeness
do not depend on the additional $T$-equivariant structure of $\sh{E}$. As our proofs below depend
on finding suitable open subsets on $X$, which must not necessarily be $T$-invariant, we will
therefore consider without loss of generality only the coarser grading by $X(H)$ rather than the
$X(G)$-grading.
\end{sub}

\begin{proposition}\label{divisorlift}
Let $D$ be a Weil divisor on $X$ and denote $\alpha \in X(H) \cong A_{d - 1}(X)$ the corresponding
class. Then $\widehat{\sh{O}(D)} \cong S(\alpha)$. In particular, $\widehat{\sh{O}}_X \cong S$.
\end{proposition}

\begin{proof}
By the isomorphism $\sh{O}(D) \cong \sh{O}(\alpha)$, we have a decomposition
as observed in Remark \ref{liftingdef}:
\begin{equation*}
\widehat{\sh{O}(D)} \cong
\bigoplus_{\beta \in X(H)} \Hom_{\sh{O}_X}(\sh{O}(\beta), \sh{O}(\alpha)) \cong
\bigoplus_{\beta \in X(H)} \Gamma(\hat{X}, \sh{O}(\alpha - \beta)) \cong S(\alpha).
\end{equation*}
\end{proof}

\begin{sub}\label{makeitlocal}
By the general properties of good quotients, any open subset $U$ of $X$ can be represented
as a good quotient $\hat{U} \goodquot H$, where $\hat{U}$ is the preimage of $U$ in $\hat{X}$ 
under the quotient map. If $U = \spec{R}$, then from the proof of Propositions \ref{firstprop}
(\ref{firstpropiii}) and \ref{divisorlift}, we can conclude that $\hat{U} = \spec{\widehat{R}}$
and $R = \widehat{R}_{(0)}$ with respect to the natural $X(H)$-grading of $\widehat{R}$.
\end{sub}

\begin{theorem}\label{coherencetheorem1}
Let $\sh{E}$ be a $T$-equivariant coherent sheaf on $X$.
\begin{enumerate}[(i)]
\item\label{coherencetheorem1i} If $\sh{E}$ is torsion free then $\widehat{\sh{E}}$ is torsion free and finitely generated.
\item\label{coherencetheorem1ii} If $\sh{E}$ is reflexive then $\widehat{\sh{E}}$ is reflexive and finitely generated.
\end{enumerate}
\end{theorem}

\begin{proof}
First we note that by the fact that $\hat{X}$ has codimension $2$ in $X$, coherence (as well as
torsion-freeness and reflexivity, respectively, see \cite[\S 1]{Hart1}) of $\sh{E}^H$ implies
that the $S$-module $\widehat{\sh{E}}$ is finitely generated
(and torsion free, respectively reflexive).
So, assertion (\ref{coherencetheorem1i}) follows from Proposition
\ref{firstprop} (\ref{firstpropiv}).

Now we prove (\ref{coherencetheorem1ii}). If $\sh{E}$ is reflexive, then by
\cite[Proposition 1.1]{Hart1}, we can choose for every point in $X$ a neighbourhood
$U = \spec{R}$ such that there exists a short exact sequence
\begin{equation*}
0 \longrightarrow \Gamma(U, \sh{E}) \longrightarrow R^n \longrightarrow F \longrightarrow 0,
\end{equation*}
where $F$ is a finitely generated, torsion free $R$-module. By \ref{makeitlocal}, we have
$U \cong \hat{U} \goodquot H$ with $\hat{U} = \spec{\widehat{R}}$ and
we can lift this sequence to
\begin{equation*}
0 \longrightarrow \Gamma(\hat{U}, \sh{E}^H) \longrightarrow \widehat{R}^n \longrightarrow G
\longrightarrow 0,
\end{equation*}
where $G$ is the homomorphic image of $S^n$ in $\widehat{F}$ and therefore torsion-free
by Proposition \ref{firstprop} (\ref{firstpropiv}).
Applying again \cite[Proposition 1.1]{Hart1}, we conclude that $\sh{E}^G$ locally reflexive and therefore reflexive. Hence, as the complement of $\hat{X}$ in $W$ has codimension at least two,
the module $\widehat{\sh{E}}$ is reflexive by \cite[Proposition 1.6]{Hart1}.
\end{proof}

We will see in Example \ref{klifting} that in general, coherence is not preserved for sheaves
with torsion.

\section{The case of toric sheaves}\label{toricsection}

We now assume that $X$ is a $d$-dimensional toric variety with associated fan
$\Delta$ and $\hat{X} \subseteq \Z^{\Delta(1)} = W$ its standard quotient presentation.
As a general reference to toric geometry we refer to \cite{oda} and \cite{Fulton};
for specifics of our setting see also \cite{perling7}.

\begin{sub}
It is customary to denote
$M := X(T) \cong \Z^d$ and $\hat{T} := G$, such that
$X(G) \cong \Z^{\Delta(1)}$. Moreover, we denote $N = M^*$ and $\Delta$ consists of strictly convex polyhedral
cones in $N \otimes_\Z \R$. We denote $\{l_\rho\}_{\rho \in \Delta(1)}$ the set of primitive
vectors of the rays in $\Delta$, which we interpret as linear forms on $M$. Elements $m \in M$ can
be considered as regular functions on $T$ and therefore as rational functions on $X$. In this case,
we write $\chi(m)$, where $\chi(m + m') = \chi(m) \chi(m')$ for any $m, m' \in M$.
We have $X(H) \cong A_{d - 1}(X)$ and the inclusion of $M$
in to $\Z^{\Delta(1)}$ yields the short exact sequence
\begin{equation*}
0 \longrightarrow M \xrightarrow{\ L \ } \Z^{\Delta(1)} \longrightarrow A_{d - 1}(X) \longrightarrow 0,
\end{equation*}
where $L$ can be represented as a matrix whose rows are formed by the $l_\rho$.
For any strictly convex rational polyhedral cone $\sigma \in \Delta$, we get an affine toric variety
$U_\sigma$ whose $M$-graded coordinate ring is given by $\ksm$ with $\sigma_M = \check{\sigma} \cap M$
and $\check{\sigma}$ denotes the dual cone of $\sigma$ in $M \otimes_\Z \R$.
Similarly, we get an exact sequence
\begin{equation*}
M \xrightarrow{\ L_\sigma \ } \Z^{\sigma(1)} \longrightarrow A_{d - 1}(U_\sigma) \longrightarrow 0,
\end{equation*}
where $L_\sigma$ is the submatrix of $L$ consisting of the rows which correspond to rays in
$\sigma(1)$.
\end{sub}

We start by recalling some facts about toric sheaves on affine toric
varieties and poset representations from \cite{perling1} and  \cite{perling7}.
Assume that $\sigma$ is a cone and $S = \K[\N^{\sigma(1)}]$
the homogeneous coordinate ring.
For any $m, m' \in M$ we write $m \leq_\sigma m'$
if and only if $m' - m \in \sigma_M$. This way we get a preordered set $(M, \leq_\sigma)$, which is
partially ordered if $\dim \sigma = d$. Equivalently, $M$ becomes a small category, where
the morphisms are given by pairs $(m, m')$ whenever $m \leq_\sigma m'$. By the preorder
$\leq_\sigma$, $M$ also becomes a topological space. Its topology is generated by open
sets $U(m) = \{m' \in M \mid m \leq_\sigma m'\}$ for every $m \in M$.

\begin{proposition}[{\cite[Prop. 5.5]{perling1}  \& \cite[Prop. 2.5]{perling7}}]
\label{repsheafequiv}
The following categories are equivalent:
\begin{enumerate}[(i)]
\item Toric sheaves on $U_\sigma$.
\item $M$-graded $\ksm$-modules.
\item Functors from $(M, \leq_\sigma)$ to the category of $\K$-vector spaces.
\item Sheaves of $\K$-vector spaces on $M$.
\end{enumerate}
\end{proposition}

Note that, given a representation $E$ of $(M, \leq_\sigma)$, the associated sheaf assigns
to any open subset $U$ of $M$ the limit $\ilim E_m$, with $m \in U$  (see \cite[Prop. 2.5]{perling7}).

Similarly, $\N^{\sigma(1)}$ induces a partial order ``$\leq$'' on $\Z^{\sigma(1)}$, which is compatible
with $\leq_\sigma$ n the following way.

\begin{lemma}
$L_\sigma(m) \leq L_\sigma(m')$ if and only if $m \leq_\sigma m'$.
\end{lemma}

\begin{proof}
We observe $L_\sigma(m) \leq L_\sigma(m') \Leftrightarrow L_\sigma(m') - L_\sigma(m) \in \N^n
\Leftrightarrow l_\rho(m' - m) \geq 0$ for every $\rho \in \sigma(1)$ $\Leftrightarrow m' - m \in
\sigma_M$.
\end{proof}

So, with respect to a fixed cone $\sigma$, it is natural to write $m \leq m'$ instead
of $L_\sigma(m) \leq L_\sigma(m')$, i.e., $m \leq m'$ if and only if
$m \leq_\sigma m'$. Moreover, for every $\uc \in \Z^n$
there exists some $m \in M$ such that $\uc \leq m$. To see this, we observe that we
always can choose some $m \in \sigma_M$ with $l_\rho(m) > 0$ for every $\rho \in \sigma(1)$ and
some integer $r > 0$ such that $\uc \leq r \cdot m$. So, for every $\uc \in \Z^{\sigma(1)}$
we obtain a nonempty open subset $U_\uc$ of $M$ which is given as
\begin{equation*}
U_\uc = \bigcup_{\uc \leq m} U(m)
\end{equation*}
By Proposition \ref{repsheafequiv}, every $M$-graded
module $E$ is equivalent to a sheaf of $\K$-vector spaces on $M$ which assigns
to every open subset $U$ of $M$ the vector space $E(U) = \ilim E_m$, where the limit is taken over
all $m \in U$. We use this to define a representation $\widehat{E}$ of $(\Z^{\sigma(1)}, \leq)$ by
setting
\begin{equation*}
\overline{E}_\uc := E(U_\uc).
\end{equation*}
By the functoriality of sheaves we have restriction maps $\overline{E}_\uc \rightarrow
\overline{E}_{\uc'}$
whenever $\uc \leq \uc'$. Hence we obtain a functor from $(\Z^{\sigma(1)}, \leq)$ to the
category
of $\K$-vector spaces and thus a $\Z^{\sigma(1)}$-graded $S$-module
$\overline{E} := \bigoplus_{\uc \in \Z^{\sigma(1)}} \overline{E}_\uc$ by Proposition \ref{repsheafequiv}.
Clearly this construction is functorial.

\begin{proposition}\label{liftingandlimes}
Denote $\widehat{E} \cong \bigoplus_{\uc \in \Z^{\sigma(1)}} \widehat{E}_\uc$ the $\Z^{\sigma(1)}$-graded
lifting of the sheaf over $U_\sigma$ associated to $E$
in the sense of Definition \ref{liftdef}. Then the modules
$\widehat{E}$ and $\overline{E}$ are naturally isomorphic. In particular,
$\widehat{E}_\uc \cong \Hom^M_\ksm(S_{(\uc)}, E)$
is naturally isomorphic to $\overline{E}_\uc$ for every $\uc \in \Z^{\sigma(1)}$.
\end{proposition}

\begin{proof}
We write $\uc = \big(c_\rho \mid \rho \in \sigma(1)\big)$.
We can consider $S_{(\uc)}$ as an $M$-graded $\ksm$-submodule of the group ring
$\K[M]$ with $S_{(\uc)} \cong \bigoplus_m \K \chi(m)$, where the sum is taken over all
$m \in M$ with $l_\rho(m) \geq -c_\rho$. Choose a minimal set of generators $s_1, \dots, s_t$
of $S_{(\uc)}$ with degrees $m_1, \dots, m_t$. Then any $M$-homogeneous homomorphism
is determined by the images of the $s_i$ in the homogeneous components $E_{m_i}$. Hence,
we can identify $\Hom^M_\ksm(S_{(\uc)}, E)$
in a natural way with a subvector space of $\bigoplus_{i = 1}^t E_{m_i}$
consisting of tuples $(e_1, \dots, e_t)$ such that $\chi(m - m_i) e_i = \chi(m - m_j) e_j$
whenever $m_i, m_j \leq_\sigma m$. But this vector space has the universal properties of the
limit $\ilim E_m$ and thus we can naturally identify it with $\widehat{E}_\uc = \ilim E_m$.
The isomorphism of the modules $\widehat{E}$ and $\overline{E}$ then follows from the
naturality of this identification.
\end{proof}

\begin{remark}
By Theorem \ref{adjointness} the lifting functor is left-exact and to any toric sheaf
$\sh{E}$ we can consider its right derived modules
\begin{equation*}
\widehat{\sh{E}} = \widehat{\sh{E}}^{(0)}, \widehat{\sh{E}}^{(1)}, \dots
\end{equation*}
By Proposition \ref{liftingandlimes} we have now a very nice interpretation of these
modules, as we can identify them degree-wise with the right derived functors of
the limit functor $\ilim$.
Right derived
limit functors $\ilim^i$ have been pioneered by Roos \cite{Roos61} and have since
been studied extensively. Roos also gives a combinatorial analog of the Cech complex
which allows in simple cases the explicit computation of the derived functors.
We can now understand the left-exactness of the lifting functor combinatorially
by the fact that the posets $\{m \in M \mid l_\rho(m) \geq -c_\rho$ for all $\rho \in
\sigma(1)\}$ are not filtered,  i.e., for any
$m, m' \in U_\uc$ there may not exist any $m'' \in U_\uc$ with $m''
\leq_\sigma m$ and $m'' \leq_\sigma m'$, which otherwise would imply the exactness of the
limit functor (see \cite[Cor. 7.2]{Jensen72}).
\end{remark}

The following example shows both that lifting in general does not preserve exactness,
and the existence of nontrivial right derived modules $\widehat{E}^{(i)}$.

\begin{example}\label{klifting}\label{kliftingcont}
Let $\sigma \subset N_\R \cong \R^3$ the the cone generated by the
primitive vectors $l_1 = (1, 0, 0)$, $l_2 = (0, 1, 0)$, $l_3 = (-1, 1, 1)$,
$l_4 = (0, 0, 1)$. Denote $\mathfrak{m} \subset \ksm$ the maximal homogeneous
ideal and consider $\K = \ksm / \mathfrak{m}$ as a simple module in degree $0$.
Now for a given $\uc = (c_1, c_2, c_3, c_4) \in \Z^4$, it is straightforward to
see that $0 \in M$ is a minimal element in $\{m \in M \mid l_i(m) \geq -c_i\}$ if and
only if
\begin{equation*}
c_1, c_3 \leq 0, c_2 = c_4 = 0 \quad \text{ or } \quad
c_1 = c_3 = 0, c_2, c_4 \leq 0.
\end{equation*}
If $\uc$ satisfies one of these conditions, then $\widehat{\K}_\uc \cong \K$ and
$\widehat{\K}_\uc = 0$ otherwise. So there is no lower bound for the $c_i$ such
that $\widehat{\K}_\uc$ vanishes and so $\widehat{\K}$ cannot be finitely generated.
We observe that $\widehat{\K}$ is Artinian and is supported precisely on those
torus orbits which get contracted to the fixed point under the quotient map
$\mathbf{A}^4_\K \rightarrow U_\sigma$.

Moreover, by Lemma \ref{divisorlift}, we have $\widehat{\ksm} \cong S$
and the long exact derived sequence of $0 \rightarrow \mathfrak{m} \rightarrow \ksm
\rightarrow \K \rightarrow 0$ starts by:
\begin{equation*}
0 \longrightarrow \widehat{\mathfrak{m}} \longrightarrow S \longrightarrow \widehat{\K}
\longrightarrow \widehat{\mathfrak{m}}^{(1)}.
\end{equation*}
By degree-wise inspection one can see that $\widehat{\mathfrak{m}} = (x_1, x_2, x_3, x_4)$,
and therefore $\widehat{\mathfrak{m}}^{(1)}$ cannot be finitely generated as well.
\end{example}

By adjointness, the lifting functor transports injective $M$-graded $\ksm$-modules to injective
$\Z^{\sigma(1)}$-graded $S$-modules.
In \cite{perling7}, codivisorial modules have been considered. For given $\uc \in
\Z^{\sigma(1)}$, such a module can be defined as $\K[-M^{\uc, I}] = \bigoplus_{m \in
-M^{\uc, I}} \K \chi(m)$, where $I$ is any subset of $\sigma(1)$ and $M^{\uc, I} =
\{m \in M \mid l_\rho(m) \geq c_\rho $ for $\rho \in I \}$. If $\uc =
L_\sigma(m)$ for some $m \in M$, then
$\K[-M^{\uc, I}]$ is an injective object in $\mksmMod$.
However, if $\K[-M^{\uc, I}]$ is not injective, the
following example shows that lifting can exhibit a more bizarre behavior than in the
previous example.

\begin{example}\label{liftex}
Let $\ksm$ be as in Example \ref{klifting} and consider the module $\K[-M^{\uc, I}]$
with $\uc = 0$ and $I = \{1, 3\}$. A similar computation as in Example
\ref{klifting} shows that
$\widehat{\K[-M^{\uc, I}]}_{(c_1, 0, c_3, 0)}
\cong \K^{1 -c_1 - c_3}$ whenever $c_1 + c_3 \leq 0$. So this module exhibits an
infinite family of graded components of any finite dimension. This shows that the
lifting functor does not respect combinatorial finiteness in the sense of \cite{perling7}.
\end{example}

\begin{sub}
Rather than limits, we can also consider colimits associated to representations of
$(M, \leq_\sigma)$. That is, for any $M$-graded $\ksm$-module $E$, there is its
colimit $\colim E_m$. As the preordered set $(M, \leq_\sigma)$ is filtered, forming
the colimit is exact. Given an $M$-graded $\ksm$-module $E \cong \bigoplus_{m \in M} E_m$,
we can associate to it the colimit $\mathbf{E} := \colim E_m$. Similarly, for the lifted
$S$-module $\widehat{E}$ we have the colimit $\widehat{\mathbf{E}} := \colim \widehat{E}_\uc$,
which is formed over the poset $(\Z^{\sigma(1)}, \leq)$.
\end{sub}

\begin{proposition}\label{EEpropiv}
In the above situation we have $\E = \widehat{\E}$.
\end{proposition}

\begin{proof}
It is easy to see that for every $\uc \in \Z^{\sigma(1)}$ we can find some $m \in M$
such that $\uc \leq m$. Conversely, for every $m \in M$ we can find some $\uc \in \Z^{\sigma(1)}$
such that $m \leq \uc$. It follows that $\colim E_m$ and $\colim \widehat{E}_\uc$ are cofinal.
\end{proof}

If $\dim \sigma < d$, then we have $m \leq_\sigma m'$ and $m' \leq_\sigma m$ whenever
$m' - m \in \sigma_M^\bot$. In particular, such a pair $(m, m')$ is an isomorphism in
the {\em category} $M$. The following proposition states that, up to natural equivalence,
we do not loose anything essential if we pass from the preordered set $(M, \leq_\sigma)$
to $M / \sigma_M^\bot$ with the induced partial order:

\begin{proposition}[{\cite[Prop. 2.8]{perling7}}]\label{orthoequiv}
Let $\Lambda \subseteq \sigma_M$ be a subgroup. Then there is an equivalence of categories
between the category of $M$-graded $\ksm$-modules
and the category of $M / \Lambda$-graded $\K[\sigma_M / \Lambda]$-modules.
\end{proposition}

Note that we state Proposition \ref{orthoequiv} in slightly greater generality than \cite{perling7}.

\begin{sub}
Now, we are ready to consider the non-affine case. Denote $\{U_\sigma\}_{\sigma \in \Delta}$
the standard covering of $X$ and $\{\hat{U}_\sigma = \spec{S_\sigma}\}_{\sigma \in \Delta}$
the corresponding cover of $\hat{X}$ given by the preimages of the $U_\sigma$.
If we take a $T$-equivariant, i.e., toric sheaf $\sh{E}$ on $X$, we see by Proposition \ref{orthoequiv}
that the $S_\sigma$-modules $\Gamma(\hat{U}_\sigma, \sh{E}^{\hat{T}})$ are naturally equivalent
to the lifts of $\Gamma(U_\sigma, \sh{E})$ to $\K[\N^{\sigma(1)}]$. In particular, it is
straightforward to check that coherence, torsion-freeness, and reflexivity are preserved
by passing back and forth between $\K[\N^{\sigma(1)}]$ and $S_\sigma$.
\end{sub}

\begin{sub}
Given a quasi-coherent sheaf $\sh{E}$ on $X$, we obtain a family of colimits
$\mathbf{E}^\sigma := \colim \Gamma(U_\sigma, \sh{E})_m$ for $\sigma
\in \Delta$. For every pair of faces $\tau, \sigma$ such that $\tau$ is a face of
$\sigma$, the restriction $\Gamma(U_\sigma, \sh{E}) \rightarrow \Gamma(U_\tau, \sh{E})$
induces a map of directed families over $(M, \leq_\sigma)$ and $(m, \leq_\tau)$,
respectively, and by the universal property of colimits we obtain an induced $\K$-linear
isomorphism $\mathbf{E}^\sigma \rightarrow
\mathbf{E}^\tau$ (see \cite[\S 5.4]{perling1}). As the face poset of $\Delta$ has
the zero cone $0$ as the unique minimal element, we can use the isomorphisms
$\mathbf{E}^\sigma \rightarrow \mathbf{E}^0$ to identify the $\mathbf{E}^\sigma$
with $\mathbf{E}^0 =: \mathbf{E}$.
For the case that $\sh{E}$ is coherent, it has been shown in \cite[\S 5.4]{perling1}
that $\dim \mathbf{E}$ equals the rank of $\sh{E}$.
We can do the same construction for
$\widehat{\sh{E}}$ and obtain a colimit $\widehat{\mathbf{E}}$, which, using Proposition
\ref{EEpropiv}, we can in a natural way identify with $\mathbf{E}$.
\end{sub}

\begin{sub}
This construction becomes most interesting for the case that $\sh{E}$ (and thus
$\widehat{\sh{E}}$ by Theorem \ref{coherencetheorem1}) is finitely generated and torsion-free.
Then the maps $\Gamma(U_\sigma, \sh{E})_M \xrightarrow{\cdot \chi(m')}
\Gamma(U_\sigma, \sh{E})_{m + m'}$ are injective for every $\sigma \in \Delta$,
$m \in M$, and $m' \in \sigma_M$. It follows that the induced maps
$\Gamma(U_\sigma, \sh{E})_m \rightarrow \mathbf{E}$ are injective as well for
every $\sigma \in \Delta$ and $m \in M$, and
analogously so for the induced maps $\widehat{\sh{E}}_\uc \rightarrow \mathbf{E}$
for $\uc \in \Z^{\Delta(1)}$. This allows a greatly condensed representation of torsion free
toric sheaves in terms of families of subvector spaces of a fixed vector space $\mathbf{E}$
which are parameterized by the family of posets $\{(M, \leq_\sigma)\}_{\sigma \in \Delta}$
(see \cite[Theorem 5.18]{perling1}).
\end{sub}

For the case of reflexive sheaves we have the following structural theorem due to
Klyachko.

\begin{theorem}[\cite{Kly90}, \cite{Kly91}, see also \cite{perling1}]\label{klythm}
The category of coherent reflexive toric sheaves on a toric variety $X$ is equivalent
to the category of vector spaces $\E$ endowed with filtrations
$0 \subseteq \cdots \subseteq E^\rho(i) \subseteq E^\rho(i + 1) \subseteq \cdots
\subseteq \E$ for $\rho \in \Delta(1)$ which are full in the sense that
$E^\rho(i) = 0$ for $i << 0$ and $E^\rho(i) = \E$ for $i >> 0$.
\end{theorem}

\begin{sub}\label{filtreconst}
Over $U_\sigma$, we observe that for a torsion free $\ksm$-module $E$ we have
$\ilim E_m$ equals the intersection
$\bigcap_{m \leq_\sigma m'} E_{m'}$ in $ \mathbf{E}$.
Therefore, given $\mathbf{E}$ and $E^\rho(i)$ for $\rho \in \sigma(1)$ as in
Theorem \ref{klythm}, one one constructs a reflexive module $E$ from this data by
setting $E = \bigoplus_{m \in M} E_m$ and $E_m = \bigcap_{\rho \in \sigma(1)}
E^\rho\big(l_\rho(m)\big) \subseteq \mathbf{E}$.
\end{sub}

By Theorem \ref{coherencetheorem1} we know that for a reflexive toric sheaf
$\sh{E}$ on $X$, its lifting $\widehat{\sh{E}}$ is reflexive as well. The fan $\widehat{\Delta}$
associated to $\hat{X}$ in general contains more cones than $\Delta$, but
we have a one-to-one correspondence between $\Delta(1)$ and $\hat{\Delta}(1)$
given by, say, $\rho \mapsto \hat{\rho}$. So we know a priori that $\sh{E}$ and
$\widehat{\sh{E}}$ are described by the same number of filtrations. The following
result shows that these filtrations (in an almost tautological sense) indeed coincide
and, moreover, that lifting is indeed ``the'' correct functor to translate reflexive
toric sheaves into $\Z^{\Delta(1)}$-graded $S$-modules.

\begin{theorem}
A toric sheaf $\sh{E}$ is coherent and reflexive if and only if $\widehat{\sh{E}}$ is coherent
and reflexive. Moreover, if $\sh{E}$ and $\widehat{\sh{E}}$ are reflexive, then they
are canonically described by the same data, i.e.,
$\widehat{\E} = \E$ and $\widehat{E}^{\hat{\rho}}(i) = E^\rho(i)$ for any $\rho \in \Delta(1)$.
In particular, lifting induces an equivalence of
categories between the category of reflexive toric sheaves on $X$, the category
of reflexive toric sheaves on $\hat{X}$, and the category of reflexive $\Z^{\Delta(1)}$-graded
$S$-modules.
\end{theorem}

\begin{proof}
The statements on coherence and reflexivity follow from Theorem \ref{coherencetheorem1}.
It suffices to consider the case that $X$ is affine, i.e., $X = U_\sigma$.
So, assume that $E$ is a reflexive $M$-graded $\ksm$-module, given by filtrations
$E^\rho(i)$ of the vector space $\E$. From this data we can construct a reflexive
$\Z^{\sigma(1)}$-graded $S$-module $F$ by setting $\mathbf{F} = \E$ and $F^{\hat{\rho}}(i)
= E^\rho(i)$. Similarly, if we start with the
reflexive $S$-module $F$, we get a reflexive $\ksm$-module $E'$ by simply identifying
the filtrations. We show that $F \cong \widehat{E}$ and $E' = F_{(0)} = E$.

The equality $E' = F_{(0)} = E$ follows from the fact that $E_m = \bigcap_{\rho \in \sigma(1)}
E^\rho\big(l_\rho(m)\big) =$ $\bigcap_{\rho \in \sigma(1)}$ $F^{\hat{\rho}}\big(l_\rho(m)\big) = F_m$
(see \ref{filtreconst}), where in the
latter equation we identify $m$ with its image $L_\sigma(m) \in \Z^{\sigma(1)}$.
Now consider $\widehat{E}_\uc$ for some $\uc \in \Z^n$. By \ref{filtreconst} we have
$\widehat{E}_\uc = \underset{\leftarrow}{\lim} E_m =
\bigcap_{\uc \leq m} E_m = \bigcap_{\uc \leq m} \bigcap_{\rho \in \Delta(1)}
E^\rho\big(l_\rho(m)\big) \subseteq \E$. Now by the fact that
the $l_\rho$ are primitive elements in $N$, we can always choose for any $\rho \in
\Delta(1)$ some $m \in M$ such that $l_\rho(m) = c_\rho$. It follows that $\widehat{E}_\uc =
\bigcap_{\rho \in \Delta(1)} E^\rho(c_\rho) = F_\uc$.

For the equivalence of categories, it suffices to
remark that for any two reflexive toric sheaves $\sh{E}, \sh{F}$, there is a natural
bijection $\Hom(\sh{E}, \sh{F}) \rightarrow \Hom(\widehat{\sh{E}}, \widehat{\sh{F}})$, as any
homomorphism of vector spaces $\E \rightarrow \mathbf{F}$ which
respects the filtrations also respects any of their intersections.
\end{proof}

\begin{remark}
For $\sh{E}$ reflexive, one can easily show that the $S$-module
$\widehat{\sh{E}}$ is isomorphic to
$(\Gamma_* \sh{E})\check{\ }\check{\ }$, the reflexive hull of $\Gamma_* \sh{E}$.
Note that more generally, if $\sh{E}$ is torsion free, then $\widehat{\sh{E}}$
does not necessarily coincide with $\Gamma_* \sh{E}$ modulo torsion.
\end{remark}

\begin{remark}
In \cite{perling7}, reflexive $M$-graded $\ksm$-modules have been investigated
in terms of the vector space arrangements underlying the associated filtrations.
Given such a module $E$, it is not difficult to see that in general not all possible
intersections are realized as the graded components $\Gamma(U_\sigma, \sh{E})_m =
\bigcap_{\rho \in \sigma(1)} E^\rho\big(l_\rho(m)\big)$.
However, for the vector space arrangement underlying the filtrations associated to
$\widehat{\sh{E}}$, all possible intersections indeed are realized this way.
In this sense, on can consider vector space arrangement in $\E$ underlying the
filtrations associated to $\widehat{E}$
as the intersection completion of the vector space arrangement underlying the
filtrations associated to $E$.
\end{remark}

\end{document}